\documentclass[12pt]{article}
\usepackage{amsmath,amsthm,amssymb}
\usepackage{mathabx,epsfig}

\usepackage[margin=1in]{geometry}
\usepackage{amsmath,amscd}
\usepackage{mathtools}

\usepackage[all]{xy}

\usepackage{thmtools}
\usepackage{thm-restate}
\usepackage{hyperref}
\usepackage{cleveref}

\newtheorem{theorem}{Theorem}[section] 
\newtheorem{definition}{Definition}[section]

\newtheorem{conjecture}{Conjecture}[section]

\newtheorem{lemma}{Lemma}[section] 
\newtheorem{proposition}{Proposition}[section]
\usepackage[small,compact]{titlesec}
\usepackage{enumerate}
\usepackage{times}
\makeatletter

\newcommand{\Rmnum}[1]{\expandafter\@slowromancap\romannumeral #1@}
\makeatother

  \newcommand{\Supp}{\operatorname{Supp}}

 \newcommand{\J}{\operatorname{J}}

 \newcommand{\HO}{\operatorname{H}}
 \newcommand{\genus}{\operatorname{genus}}

  \newcommand{\Res}{\operatorname{Res}}

  \newcommand{\End}{\operatorname{End}}
 \newcommand{\Aut}{\operatorname{Aut}}

\newcommand{\Gal}{\operatorname{Gal}}
\newcommand{\rank}{\operatorname{rank}}

\usepackage{xcolor}
\usepackage{titlesec}
\titleformat{\section}[block]{\color{black}\Large\bfseries\filcenter}{}{1em}{}
\titleformat{\subsection}[hang]{\bfseries}{}{1em}{}
\theoremstyle{remark}
\newtheorem{remark}{Remark}[section]
\begin{document}
\title{The Arithmetic of Curves Defined by Iteration}
\author{Wade Hindes\\
Department of Mathematics, Brown University\\
Providence, RI 02912\\
Contact Information: (810) 434-3401, whindes@math.brown.edu}
\date{\today}
\maketitle
\begin{abstract} We show how the size of the Galois groups of iterates of a quadratic polynomial $f$ can be parametrized by certain rational points on the curves $C_n:y^2=f^n(x)$ and their quadratic twists. To that end, we study the arithmetic of such curves over global and finite fields, translating key problems in the arithmetic of polynomial iteration into a geometric framework. This point of view has several dynamical applications. For instance, we establish a maximality theorem for the Galois groups of the fourth iterate of quadratic polynomials $x^2+c$, using techniques in the theory of rational points on curves. Moreover, we show that the Hall-Lang conjecture on integral points of elliptic curves implies a Serre-type finite index result for these dynamical Galois groups, and we use conjectural bounds for the Mordell curves to predict the index in the still unknown case when $f(x)=x^2+3$. Finally, we provide evidence that these curves defined by iteration have geometrical significance, as we construct a family of curves whose rational points we completely determine and whose geometrically simple Jacobians have complex multiplication and positive rank.            
\end{abstract}
\section{1. Introduction}{\label{Intro}}

\indent \indent  The arithmetic properties of iterated rational maps: heights, integer points in orbits, preperiodic points etc. provide many interesting problems, both directly and by analogy, in classical arithmetic geometry. In this paper, we study another object straddling both geometry and dynamics, the arboreal representation attached to a rational function. Specifically, in the case of quadratic polynomials $f=f_c(x)=x^2+c$, we translate key problems regarding the size of the Galois groups of iterates of $f$ into a geometric framework.

To state these results, we fix some notation. For $f=x^2+c$ let $G_n(f)=\Gal(f^n)$ and let $\mathbb{T}_n$ denote the set of roots of $f, f^2,\dots ,f^n$ together with $0$. Furthermore, set
\begin{equation}{\label{Arboreal}}
\mathbb{T}_\infty:=\bigcup _{n \geq 0} f^{-n}(0)\;\;\text{and}\;\; G_\infty=\varprojlim G_n(f).
\end{equation}  
If $f$ is irreducible, then $\mathbb{T}_n$ carries a natural $2$-ary rooted tree structure: $\alpha,\beta\in\mathbb{T}_n$ share an edge if and only if $f(\alpha) =\beta$. Furthermore, as $f$ is a polynomial with rational coefficients, $G_n(f)$ acts via graph automorphisms on $\mathbb{T}_n$. Hence, we have injections $G_n \hookrightarrow \Aut(\mathbb{T}_n)$ and $G_\infty \hookrightarrow \Aut(\mathbb{T}_\infty)$ called the \emph{arboreal representations} associated to $f$. A major problem in arithmetic dynamics, most notably because of its application to density questions in orbits \cite{Jones2}, is to study the size of these images. For a nice exposition on the subject, as well as the formulation for rational functions $\phi\in\mathbb{Q}(x)$, see \cite{R.Jones}. 

For a fixed stage $n$, a natural question to ask is which rational values of $c$ supply a polynomial $x^2+c$ whose $n$-th iterate is the first to have smaller than expected Galois group. When $n=4$, the only examples up to a very large height are $c=2/3$ and $c=-6/7$. Moreover, we prove that there are no such integer values (in contrast to the $n=3$ case in \cite{Me}) and formulate questions for larger $n$. Specifically, let \[S^{(n)}=\Big\{c\in\mathbb{Q}\Big\vert \;\;\big\vert\Aut(\mathbb{T}_{n-1}):G_{n-1}(f_c)\big\vert=1\;\; \text{and}\;\;\; \big\vert\Aut(\mathbb{T}_n):G_n(f_c)\big\vert>1\Big\}.\]
Then we use techniques in the theory of rational points on curves: Chabauty's method, unramified coverings, the Mordell-Weil sieve, and bounds on linear forms in logarithms, to deduce the following maximality result for the fourth iterate these quadratic polynomials.

\begin{restatable}{thm}{SmallFourth}\label{thm:SmallFourth}  Let $f_c(x)=x^2+c$. Then all of the following statements hold. 
\begin{enumerate} 
\item $S^{(4)}\cap\mathbb{Z}=\emptyset$. That is, if $c$ is an integer and $G_3(f_c)\cong\Aut(\mathbb{T}_3)$, then $G_4(f_c)\cong\Aut(\mathbb{T}_4)$.  
\item If $c\neq3$ is an integer and $G_2(f_c)\cong\Aut(\mathbb{T}_2)$, then $G_4(f_c)\cong\Aut(\mathbb{T}_4)$. 
\item If the curve $F_2: y^2=x^6 + 3x^5 + 3x^4 + 3x^3 + 2x^2 + 1$ has no rational points of height greater that $10^{100}$, then $S^{(4)}=\{2/3,-6/7\}$. 
\end{enumerate} 
\end{restatable}  

In the proof of this and subsequent theorems, the key objects which parametrize the size of these dynamical Galois groups are the hyperelliptic curves \[C_n: y^2=f^n(x)\;\; \text{and}\;\; B_n: y^2=(x-c)\cdot f^n(x).\] Adding to the evidence for open image theorems in dynamics (see \cite{Me} and \cite{PrimitiveDivisors}), we use these curves defined by iteration and some standard conjectures in arithmetic geometry to prove the following theorem. 

\begin{restatable}{thm}{HallLang} Let $f(x)=x^2+c$\; for some integer $c$. If $c\neq-2$ and $-c$ is not a square, then all of the following statements hold:  
\begin{enumerate} 
\item The Hall-Lang conjecture implies that $\big|\Aut(\mathbb{T}_{\infty}):G_{\infty}(f)\big|$ is finite. 
\item If the weak form of Hall's conjecture for the Mordell curves holds with $C=100$ and $\epsilon=4$, then when $f(x)=x^2+3$, we have that $\big|\Aut(\mathbb{T}_\infty):G_{\infty}(f)\big|=2$.
\end{enumerate} 
\label{thm:Hall-Lang}
\end{restatable}
\begin{remark} This result is analogous to a theorem of Serre for non CM elliptic curves \cite{B-J}. Moreover, this analogy is particularly interesting since when $c=-2$, the relevant family of curves actually has CM; see Theorem \ref{thm:Chebychev} below. Assuming the ABC conjecture holds over $\mathbb{Q}$, this ``Finite Index Conjecture" was proven independently by the author in \cite{Me} and by Gratton, Nguyen, and Tucker in \cite{PrimitiveDivisors} . 
\end{remark}    

We begin section $2$ by discussing some general arithmetic properties of $C_n$ and $B_n$. Specifically, we address problems related to the torsion subgroups and simple factors of their Jacobians. To do this, we extract information from the specific case when $c=-2$ and $f$ is a Chebychev polynomial. In this case, the curves in this family have some very special properties. In pariticular, let $J(B_n)$ be the Jacobian of $B_n$. Then we have the following theorem.  
\begin{restatable}{thm}{Chebychev}
\label{thm:Chebychev} When $f(x)=x^2-2$, all of the following statements hold: 
\begin{center}\begin{enumerate}
\item$J(B_n)$ is an absolutely simple abelian variety  that has complex multiplication by $\mathbb{Q}(\zeta+\zeta^d)$, where $\zeta$ is a primitive $2^{n+2}\text{-th}$ root of unity and $d=2^{n+1}-1$. 
\item Consider $B_n/\,\mathbb{F}_p$ and let $\chi(B_n,t)$ be the characteristic polynomial of Frobenius. 
\begin{enumerate} 
\item If $p\equiv5\bmod{8}$, then $\chi(B_n,t)=t^{2^n}+p^{2^{n-1}}$ for all $n\geq1$. 
\item If $p\equiv3\bmod{8}$, then $\chi(B_n,t)=t^{2^n}+p^{2^{n-1}}$ for all $n\geq2$.
\end{enumerate} 
\item $J(B_n)(\mathbb{Q})_{Tor}\cong\mathbb{Z}/2\mathbb{Z}$\; for all $n\geq1$. Therefore, $\rank\big(J\left(B_n\right)(\mathbb{Q})\big)\geq1$\; for all $n\geq2$. Furthermore, we have that\; $\rank\big(J\left(C_n\right)(\mathbb{Q})\big)\geq n-2$\; for all $n\geq1$.   
\item $B_n(\mathbb{Q})=\{\infty,(-2,0),(0,\pm{2})\}$ for all $n\geq2$. 
\end{enumerate}
\end{center}
\end{restatable}
Note that we have determined the rational points on a family of curves whose Jacobians are of  positive rank and geometrically simple, a usually difficult task. The key point is that our polynomial $f=x^2-2$ is equipped with a rational cycle, and we will discuss generalizations of this construction to other polynomials with similar dynamical properties; see (\ref{fact}). Furthermore, as a corollary, we obtain the factorization (into simple factors) of the Jacobians of $C_n$ for a large class of quadratic polynomials.   

\begin{restatable}{cor}{half} \label{cor:half}If $f(x)=x^2+ax+b\equiv x^2-2\bmod{p}$\; for some $p\equiv\pm{3}\bmod8$, then the decomposition \[J(C_n)\sim J(B_1)\times J(B_2)\dots\times J(B_{n-1})\] is indecomposable over \;$\mathbb{Q}$. In particular, this decomposition is indecomposable for at least half of all polynomials $f_c(x)=x^2+c$.  
\end{restatable}  

We close with a discussion of certain Galois uniformity questions, analogous to those for preperiodic points of rational polynomials; see \cite{Poonen2} and \cite{Ingram}. For the sake of completeness, all Galois groups were computed with Sage \cite{Sage}, and the descent calculations were carried out with Magma \cite{Magma}. Moreover, $\Res(f,g)$ will denote the resultant of two polynomials $f$ and $g$ throughout.    

\section{2. Arithmetic Properties of Curves Defined by Iteration}{\label{AQI}}
\indent\indent Let $f=f_c(x)=x^2+c$. As mentioned in the introduction, the size of the Galois group of $f^n$ is encoded in the existence of certain rational points on the curve 
\begin{equation}{\label{C_n}}C_n: =\big\{(x,y)\;\vert \;y^2=f_c^n(x)\big\},
\end{equation} 
and its quadratic twists \big(for now, we bracket the discussion of this correspondence and take it up in section \ref{Classifying Galois}\big). In particular, to understand the Galois groups as we iterate $f$, we must study the arithmetic of these curves. As a first step, we analyze the torsion subgroups and simple factors of their Jacobians in this section.  

When studying the rational points on a curve of large genus, one often attempts to find a map to a curve of lower genus, where the typical arithmetic procedures are more easily carried out. Note that by iterating $f$ we obtain maps 
\begin{equation}{\label{C_n}} C_n\xrightarrow{f} C_{n-1}\xrightarrow{f} C_{n-2}\xrightarrow{f}\dots\xrightarrow{f} C_1.  
\end{equation} 

However, in order to completely decompose the Jacobians of $C_n$ (and perhaps compute their endomorphism rings), it would be better to find maps to curves whose Jacobians are simple. Fortunately, for $m<n$ we also have the coverings $\pi_m:C_n\rightarrow B_m$\; given by
\begin{equation}{\label{B_n}}
B_m:=y^2=(x-c)\cdot f_c^m(x)\;\;\;\text{and}\;\;\; \pi_m(x,y)=\Big(f_c^{n-m}(x),\;y\cdot f_c^{n-m-1}(x)\Big).
\end{equation}

\begin{remark}Similar maps and curves, $\pi_m$ and $B_m$, can be constructed for all quadratic polynomials $f=x^2+bx+c$ simply by completing the square; see \cite{Me}. \end{remark}
 From here we can deduce that the Jacobian of $C_n$ decomposes  as one might expect from our setup. For simplicity, we adopt the notation $J(C)$ for the Jacobian of any curve $C$ throughout. Furthermore, we write $A\sim B$ when $A$ and $B$ are isogenous abelian varieties.     
  \begin{proposition}{\label{prop: decomp}} Let $f(x)=f_c(x)=x^2+c$. If $f^n$ is a separable, then $C_n$ is nonsingular and 
 \begin{equation}{\label{decomp}} J(C_n)\sim J(B_1)\times J(B_2)\times\dots\times J(B_{n-1})
 \end{equation} 
 for all $n\geq2$. In particular, we have such a decomposition when $f$ is irreducible.  
\begin{proof} We will proceed by induction on $n$. If $n=2$, then both $C_2$ and $B_1$ have genus 1. Moreover, $\pi_1: C_2\rightarrow B_1$ is nonconstant ($f$ has degree two) and hence the induced map on Jacobians must be an isogeny. 

Now for the general case. We fix $n$ and let $f_m: C_n\rightarrow C_m$ be the map $(x,y)\rightarrow\big(f^{n-m}(x),\;y\big)$. The maps $f_{n-1}$ and $\pi_{n-1}$ induce maps \[\phi=(\pi_{n-1},f_{n-1}): C_n\rightarrow B_{n-1}\times C_{n-1}\;\;\;\text{and}\;\;\; \phi_*: J(C_n)\rightarrow J(B_{n-1})\times J(C_{n-1}).\] We show that $\phi$ induces an isomorphism $\phi^*$ on the space of regular differentials on $C_n$ and on $B_{n-1}\times C_{n-1}$, from which it follows that $\phi_*$ is an isogeny.

Note that $\genus(C_n)=2^{n-1}-1$ and $\genus(B_m)=2^{m-1}$. It follows that for $0\leq i\leq2^{n-2}-1$, the set $\big\{\frac{x^i}{y}dx\big\}$ is a basis of $\HO^0\big(B_{n-1},\Omega\big)$. Similarly, $\big\{\frac{x^j}{y}dx\big\}$ is a basis of $\HO^0\big(C_{n-1},\Omega\big)$, for $0\leq j\leq2^{n-2}-2$. One computes that
\begin{equation}{\label{d1}}
\pi_{n-1}^*\left(\frac{x^i}{y}dx\right)=\frac{f(x)^i}{x\cdot y}d\big(f(x)\big)=2\frac{f(x)^i}{y}dx,
\end{equation}      
and 
\begin{equation}{\label{d2}} f_{n-1}^*\left(\frac{x^j}{y}dx\right)=\frac{f(x)^j}{y}d\big(f(x)\big)=\frac{2\cdot f(x)^j\cdot x}{y}dx.
\end{equation} 
Since \[\dim_{\mathbb{Q}}\Big(\HO^0\big(C_n,\Omega\big)\Big)=2^{n-1}-1=2^{n-2}+2^{n-2}-1=\dim_{\mathbb{Q}}\Big(\HO^0\big(B_{n-1},\Omega\big)\Big)+\dim_{\mathbb{Q}}\Big(\HO^0\big(C_{n-1},\Omega\big)\Big),\] it suffices to show that $\phi^*$ is surjective, to infer that it is an isomorphism. To do this, it is enough to show that $\big\{\frac{x^i}{y}dx\big\}$ inside $\HO^0(C_{n}, \Omega)$ is in the span of the images of $\pi_{n-1}^*$ and $f_{n-1}^*$ for all ${0\leq i\leq2^{n-1}-2}$. To establish this, we again precede by induction. 

Note that by (\ref{d1}) and (\ref{d2}), we have that 
\[\frac{dx}{y}=\frac{1}{2}\cdot\frac{2dx}{y}=\frac{1}{2}\cdot \pi_{n-1}^*\left(\frac{dx}{y}\right),\;\;\;\text{and}\;\;\;\; \frac{xdx}{y}=\frac{1}{2}\cdot\frac{2x\cdot dx}{y}=\frac{1}{2}\cdot f_{n-1}^*\left(\frac{dx}{y}\right).\]
As for the inductive step, suppose that $\frac{x^{i}dx}{y}$ is in the span of $\pi_{n-1}^*$ and $f_{n-1}^*$ for all $i\leq t-1$. Furthermore, assume that $t$ is even and write $t=2k$. If we write $f(x)^k=x^t+\sum_{i=0}^{t-1}c_ix^i$, then \[\frac{x^tdx}{y}=\frac{1}{2}\cdot\pi_{n-1}^*\left(\frac{x^kdx}{y}\right)-\sum_{i=0}^{t-1}c_i\frac{x^idx}{y}.\] However, the tail sum is in the span of $\pi_{n-1}^*$ and $f_{n-1}^*$ by the induction hypothesis (note that $t$ is even and $t\leq2^{n-1}-2$ implies that $k\leq2^{n-2}-1$, which is needed to force $\frac{x^kdx}{y}\in\HO^0\big(B_{n-1},\Omega\big)$ as desired). 

Similarly, if $t=2k+1$ is odd, then write $f(x)^{k}= x^{2k}+\sum_{i=0}^{t-2}s_ix^i$. We see that 
\[ \frac{x^tdx}{y}=\frac{x\cdot f(x)^kdx}{y} -\sum_{i=0}^{t-2}s_i\frac{x^{i+1}dx}{y}= \frac{1}{2}f_{n-1}^*\left(\frac{x^kdx}{y}\right)-\sum_{i=0}^{t-2}s_i\frac{x^{i+1}dx}{y}.\]  
Again, the tail sum is in the intended span. To see this, note that $t\leq2^{n-1}-2$ implies that $k\leq2^{n-2}-3$, and hence $\frac{x^kdx}{y}$ is in $\HO^0(C_{n-1}, \Omega)$. 

The argument above establishes that $J(C_n)\sim J(B_{n-1})\times J(C_{n-1})$. However, by induction we deduce that \[J(C_n)\sim J(B_{n-1})\times J(B_{n-2})\times\dots \times J(B_1)\] as claimed.     
\end{proof}  
\end{proposition}
In  Corollary \ref{cor:half}, we show that for many values of $c$ and every $m$, the Jacobian $J(B_m)$ of $B_m$ is simple. To do this, we extract information from the special case when $c=-2$. In this situation, $f=x^2-2$ is a Chebyshev polynomial of degree $2$, often denoted $T_2$ elsewhere in the literature \cite[1.6]{SilvDyn}. More generally, $f^n$ is the Chebyshev polynomial $T_{2^n}$ of degree $2^n$.

We consider the Chebyshev polynomials $T_d$ as characterized by the equations 
\begin{equation}{\label{Cheb}} 
T_d(z+z^{-1})=z^d+z^{-d}\;\;\; \text{for all}\;\; z\in\mathbb{C}^*,
\end{equation} 
and $T_d$ is known to be a degree $d$ monic polynomial with integer coefficients. The classical Chebyshev polynomials $\tilde{T}_d$ were defined in the following way: \[\text{If we write}\;\; z=e^{it},\;\;\text{then}\;\;\; \tilde{T}_d(2\cos(t))=2\cos(dt),\]  though we use the first characterization in (\ref{Cheb}), where $T_d$ is monic. For a complete discussion of these polynomials, see \cite[1.6]{SilvDyn}. 

It has long been known that the dynamical behavior of the Chebychev polynomials is particularly simple. We will harness this simplicity to deduce strong conclusions about the curves $B_n$ in this case. However, the key insight is that a polynomial does not have to actually be a Chebychev polynomial for parts of this analysis to work, but simply reduce to $x^2-2$ modulo some prime $p\equiv\pm{3}\bmod{8}$. 

In particular, we obtain arithmetic information for a large class of quadratic polynomials. Our first result in the Chebychev case is the following. 

\begin{theorem}{\label{over finite}} Let $f(x)=x^2-2$ and consider the curves $B_n/\;\mathbb{F}_p$. 
\begin{enumerate} 
\item If $p\equiv5\bmod{8}$, then $B_n$ has characteristic polynomial $\chi(B_n,t)=t^{2^n}+p^{2^{n-1}}$ for all $n\geq1$. 
\item If $p\equiv3\bmod{8}$, then $B_n$ has characteristic polynomial $\chi(B_n,t)=t^{2^n}+p^{2^{n-1}}$ for all $n\geq2$. 
\end{enumerate} 
In particular, $J(B_n)(\mathbb{F}_p)\cong\mathbb{Z}/(p^{2^{n-1}}+1)\mathbb{Z}$\; and $J(B_n)$ is supersingular for the $n$ and $p$ given above.   

\begin{proof} To compute $\chi(B_n,t)$, we consider the auxiliary curves 
\begin{equation}{\label{aux curves}} \mathfrak{C}_n:y^2=x\cdot (x^{2^n}+1)\;\;\text{and}\;\;\; B_n^{\pm}:y^2=(x\pm2)\cdot T_{2^n}(x).
\end{equation} 
 Note that $\mathfrak{C}_{n+1}$ is equipped with the maps 
\begin{equation}{\label{aux maps}}
\phi_{\pm}:\mathfrak{C}_{n+1}\rightarrow B_n^{\pm}, \;\; \phi_{\pm}(x,y)=\left( x+\frac{1}{x}\;,\;\frac{(x\pm1)\cdot y}{x^{2^{n-1}+1}}\;\right),
\end{equation} 
and that $B_n^{\pm}/\,\mathbb{F}_p$ are nonsingular for every odd prime. The nonsingularity follows from the fact that $T_2=f$ is critically finite: $\{f(0),f^2(0),f^3(0),\dots\}=\{\pm2\}$. Hence the discriminant of $(x\pm2)\cdot f^n(x)$ is a power of $2$. In general, the discriminant $\Delta_m$ of $f_c^m$ for any quadratic polynomial $f_c=x^2+c$ is 
\begin{equation}{\label{discriminant}}
\Delta_m=\pm\Delta_{m-1}^2\cdot2^{2^m}\cdot f_c^m(0);
\end{equation} 
 see \cite[Lemma 2.6]{Jones2}. Before proceeding with the proof of Theorem \ref{over finite}, we are in need of a few lemmata. 
\begin{lemma}{\label{basic}} Let $n$ be any positive integer. If $m<2^n$ and $p\equiv\pm{3}\bmod{8}$, then the field $\mathbb{F}_{p^m}$ contains an element $\alpha$ satisfying $\alpha^{2^{n+1}}=1$, which is not a square in $\mathbb{F}_{p^m}$. 
\begin{proof} Write $q=p^m$ and suppose that $m=2^t$. Notice that $\mathbb{F}_q$ contains an element $\alpha$ of order $2^{t+2}$, since $p^{2^t}-1$ is divisible by $2^{t+2}$ and $\mathbb{F}_q^*$ is cyclic. To see this, write \[(p^{2^t}-1)=(p^{2^{t-1}}-1)\cdot(p^{2^{t-1}}+1)=(p-1)\cdot(p+1)\cdot(p^2+1)\dots\cdot(p^{2^{t-1}}+1),\] inductively. As either $p+1$ or $p-1$ is $\equiv 0\bmod{4}$ and every other term in the product is even, we see that $p^{2^t}-1$ is divisible by $2^{t+2}$. 

However, $p\equiv\pm{3}\bmod{8}$ implies that no higher power of $2$ can divide the product. Hence $\alpha$ is not a square in $\mathbb{F}_q$. Finally, the conditions on $m$ force $t\leq n-1$. Therefore, $\alpha^{2^{n+1}}=1$ as desired.      

In general we may write $m=2^t\cdot a$ for some odd $a$. By applying the result in the $2$-powered case to the subfield $\mathbb{F}_{p^{2^t}}\subset\mathbb{F}_q$, we may find an element $\alpha\in\mathbb{F}_{p^{2^t}}$ with the desired properties. Note that if such an element $\alpha$ were a square in $\mathbb{F}_q$, then it must be a square in $\mathbb{F}_{p^{2^t}}$ (otherwise there would be a proper quadratic extension of $\mathbb{F}_{p^{2^t}}$ contained in $\mathbb{F}_q$, contradicting the fact that $a$ is odd). However, as was the case above, the fact that $p\equiv\pm{3}\mod{8}$ implies that $\alpha$ is not a square in $\mathbb{F}_{p^{2^t}}$.   
\end{proof}
\end{lemma}
\begin{lemma}{\label{bijection}} If $v_2(p+1)=k$, then $\#B_n^{+}(\mathbb{F}_q)=\#B_n^{-}(\mathbb{F}_q)$ for all $n\geq k$ and all $q=p^t$.  
\begin{proof} Notice that when $q=p^{2t}$ or $p\equiv1\bmod{4}$, the claim easily follows: choose $\alpha$ in $\mathbb{F}_q$ such that $\alpha^2=-1$. Then $(x,y)\rightarrow(-x, \alpha y)$ is a bijection from $B_n^+(\mathbb{F}_q)$ to $B_n^-(\mathbb{F}_q)$. 

When $q=p^{2t+1}$ we define the bijections 
\begin{equation}{\label{pi's}}
\pi_{+}: B_n^{+}(\mathbb{F}_q)\rightarrow\ B_n^{-}(\mathbb{F}_q)\;\; \text{and}\;\; \pi_{-}: B_n^{-}(\mathbb{F}_q)\rightarrow\ B_n^{+}(\mathbb{F}_q)
\end{equation} 
 using the following strategy: for $\pi_+$ we take $(a,b)\in B_n^+(\mathbb{F}_q)$ and pullback $\phi_+$ defined in (\ref{aux maps}) to a point $(w,y)\in\mathfrak{C}_n(\mathbb{F}_{q^2})$. Next, apply the endomorphism $(w,y)\rightarrow(\zeta^2w,\zeta y)$ on $\mathfrak{C}_n(\mathbb{F}_{q^2})$ for a suitably chosen root of unity $\zeta$. Finally, apply $\phi_-$ to get a point on $B_n^-(\mathbb{F}_q)$. Define $\pi_-$ in a similar fashion.    

Specifically, let $a\in\mathbb{F}_q$ and write $a=w+w^{-1}$ for some $w\in\mathbb{F}_{q^2}^*$. Moreover, since $v_2(p+1)=k$, we may choose a primitive $2^{k+1}$-th root of unity $\zeta\in\mathbb{F}_{p^2}\subseteq\mathbb{F}_{q^2}$. Define the maps $\pi_+$ and $\pi_-$ as follows:    
\begin{displaymath}
   \pi_+(a,b) := \left\{
     \begin{array}{llrr}
       (2,0) & : \;(a,b)=(-2,0)\\
       \\
       \Big(a,\frac{(w-1)}{w+1}\cdot b\Big)& : \; a\neq-2, \;a^2-4\in\mathbb{F}_q^2\\
       \\
       \left(\zeta^2w+(\zeta^2w)^{-1}, \frac{\zeta^2w-1}{w+1}\cdot\frac{-1}{\zeta}\cdot b\right) & : \; a\neq-2,\; a^2-4\notin\mathbb{F}_q^2\\ 
       \\
       \infty & : \;(a,b)=\infty
     \end{array}
   \right.
\end{displaymath} and, 
\begin{displaymath}
   \pi_-(a,b) := \left\{
     \begin{array}{llrr}
       (-2,0) & : \;(a,b)=(2,0)\\
       \\
       \left(a,\frac{(w+1)}{w-1}\cdot b\right)& : \; a\neq-2, \;a^2-4\in\mathbb{F}_q^2\\
       \\
       \left(\zeta^2w+(\zeta^2w)^{-1}, \frac{\zeta^2w-1}{w+1}\cdot-\zeta\cdot b\right) & : \; a\neq-2,\; a^2-4\notin\mathbb{F}_q^2\\ 
       \\
       \infty & : \;(a,b)=\infty
     \end{array}
   \right.
\end{displaymath}     
One easily checks that $\pi_+$ and $\pi_-$ are inverses, that $\pi_+\big(B_n^+(\mathbb{F}_q)\big)\subset B_n^-(\mathbb{F}_{q^2})$, and that $\pi_-\big(B_n^-(\mathbb{F}_q)\big)\subseteq B_n^+(\mathbb{F}_{q^2})$. The content which remains to be checked is that the respective images of $\pi_+$ and $\pi_-$ are in fact defined over $\mathbb{F}_q$. We complete the argument for $\pi_+$ only, as the argument for $\pi_-$ is identical.   

Suppose that $a^2-4\in\mathbb{F}_q^2$ and that $(a,b)\in B_n^+(\mathbb{F}_q)$. Since $w+w^{-1}=a$ (or equivalently $w^2-aw+1=0$), the quadratic formula tells us that $w\in\mathbb{F}_q$. Hence $\frac{w+1}{w-1}\cdot b\in\mathbb{F}_q$ and $\pi_+(a,b)\in B_n^-(\mathbb{F}_q)$ as claimed. 

On the other hand, if $a^2-4\notin\mathbb{F}_q^2$, then the frobenius map $x\rightarrow x^q$ acts nontrivially on $w$, and must send it to $w^{-1}$, the only other root of $x^2-ax+1=0$. Note that \[(\zeta^2)^q=(\zeta^2)^{(p^{2t+1})}=((\zeta^2)^{p^{2t}})^p=(\zeta^2)^p=\zeta^{-2},\] since $\zeta^2\in\mathbb{F}_{p^2}$ (hence fixed by applying $x\rightarrow x^p$ an even number of times) and $2\cdot (p+1)\equiv0\bmod{2^{k+1}}$. It follows that $(\zeta^2w)^q=(\zeta^2w)^{-1}$ and $(\zeta^2w)+(\zeta^2w)^{-1}\in\mathbb{F}_q$.   

For the second coordinate, we compute that \[ \left(\frac{\zeta^2w-1}{w+1}\cdot\frac{-1}{\zeta}\cdot b\right)^q=\frac{\frac{1}{\zeta^2w}-1}{\frac{1}{w}+1}\cdot\frac{-1^q}{\zeta^q}\cdot b^q=\frac{1}{\zeta^{q+2}}\cdot\frac{\zeta^2w-1}{w+1}\cdot b\;.\] It suffices to show that $\zeta^{q+1}=-1$. To see this, write $q=p^{2t+1}$ and use the fact that $(\zeta)^{p^{2t}}=\zeta$, since $\zeta\in\mathbb{F}_{q^2}$. In particular, we compute that $\zeta^{q+1}=\zeta(\zeta^{p^{2t}})^p=(\zeta)^{p+1}=-1$ as desired. Hence $\pi_+(a,b)\in B_n^-(\mathbb{F}_q)$.       
\end{proof}
\end{lemma}
Now for the proof of Theorem \ref{over finite}. One checks that $\phi_{\pm}$ is a quotient map for the involution \[\psi_{\pm}(x,y)=\left(\frac{1}{x}\;,\frac{\pm y}{x^{2^n+1}}\right),\] of $\mathfrak{C}_{n+1}$ defined in (\ref{aux curves}). If $G=\Aut(\mathfrak{C}_{n+1})$, then a result of Kani and Rosen \cite{Decomp} on the equivalence of idempotents in the group algebra $\mathbb{Q}[G]$ implies that the Jacobian of the curve $\mathfrak{C}_{n+1}$ has a decomposition: 
\begin{equation}{\label{Rosen Decomp}} 
J(\mathfrak{C}_{n+1})\sim J(B_n^+)\times J(B_n^-)
\end{equation} 
over the rationals. Moreover, since every odd prime is a prime of good reduction, we have a similar decomposition of $J(\mathfrak{C}_n)$ over $\mathbb{F}_p$. 

Therefore, it suffices to compute the characteristic polynomial of Frobenius for $\mathfrak{C}_{n+1}$ to find that of $B_n$. In keep with our earlier notation, we denote this characteristic polynomial by $\chi(\mathfrak{C}_{n+1},t)$. We refer the reader to \cite[\S14.1]{Coh-Finite} for the formulas relating the coefficients of this polynomial to the number of points on the curve over $\mathbb{F}_q$. 

Note that the genus of $\mathfrak{C}_{n+1}$ is $g=2^n$, and so to compute the coefficients of the Euler polynomial of $\mathfrak{C}_{n+1}$, we need to find \[N_m=\#\mathfrak{C}_{n+1}(\mathbb{F}_q)=q+1+\sum_{x\in\mathbb{F}_{q}}  \left(\frac{x}{q}\right) \left(\frac{x^{2^{n+1}}+1}{q}\right),\] for $q=p^m$ and $m\leq2^{n}$. To do this, suppose that $m<2^{n}$ and apply Lemma \ref{basic} to find a non-square $\alpha\in\mathbb{F}_{p^m}$ satisfying $\alpha^{2^{n+1}}=1$. We compute that
\begin{small}
\begin{equation}{\label{character sum}} 
N_m-\;(q+1)=\sum_{x\in S}\sum_{i=1}^{2^{n+1}}\left(\frac{x\cdot \alpha^i}{q}\right)\left(\frac{(\alpha^i\cdot x)^{2^{n+1}}+1}{q}\right)=\sum_{x\in S}\left(\sum_{i=1}^{2^{n+1}}\left(\frac{\alpha^i}{q}\right)\right) \left(\frac{x}{q}\right) \left(\frac{x^{2^{n+1}}+1}{q}\right)=0,
\end{equation} 
\end{small} 
where $S$ is a set of coset representatives for $\mathbb{F}_q^*/\langle\alpha\rangle$; the final equality follows from the fact that $\left(\frac{\alpha}{p}\right)=-1$.

 It is known that $\chi$ can be expressed as \[\chi(t)=t^{2g}+a_1t^{2g-1}+\dots+a_g+pa_{g-1}+\dots+p^g,\] where for $i\leq g$, one has that 
 \begin{equation}{\label{coefficients}}
 ia_i=(N_i-p^i-1)+(N_{i-1}-p^{i-1}-1)a_i+\dots +(N_1-p-1)a_{i-1}.
\end{equation}  
  This recurrence relation for the coefficients $a_i$ follows from Netwon's Formula expressing the elementary power polynomials in terms of the elementary symmetric functions; see \cite[p.619]{DF}. 
 
 It follows from the character sum on (\ref{character sum}) and the expression in (\ref{coefficients}) that 
\begin{equation}{\label{characteristic}} \chi(\mathfrak{C}_n,t)=t^{2^{n+1}}+at^{2^n}+p^{2^n}\;\;\text{and}\;\; 2^n\cdot a=N_{2^n}-p^{2^n}-1.
\end{equation}  
Note that the Hasse-Weil bound implies that $a\leq2p^{2^{n-1}}$. We show that this is an equality.

From Lemma \ref{bijection}, it follows that $J(\mathfrak{C}_{n+1})\sim J(B_n^{+})^2$ and $\chi(\mathfrak{C}_n,t)=\chi(B_n^+,t)^2$ for the $n$ designated in Theorem \ref{over finite}. If we write \[a= 2p^{2^{n-1}}+2b_1^2\cdot p^{2^{n-1}-1}+\dots +b_{2^{n-2}}^2,\] where the $b_i$'s are the coefficients of the characteristic polynomial of $J(B_n)$, then the bound on $a$ implies that \[2b_1^2\cdot p^{2^{n-1}-1}+\dots +b_{2^{n-2}}^2=0.\] We conclude that $b_i=0$ for all $i$, and that $\chi(B_n,t)=t^{2^{n}}+p^{2^{n-1}}$ as claimed. The group structure of $J(B_n)(\mathbb{F}_p)$ follows from a theorem of Zhu, which may be found in \cite[\S45]{CurvesFinite}  
\end{proof} 
\end{theorem}
As mentioned in the introduction, we can extract global information for a large class of quadratic polynomials from the local data in the Chebychev case. We restate this information here.    
\half*
\begin{remark} We can also use Theorem \ref{over finite} to extract global torsion data. For instance, let $f(x)=x^2+63$. Then one can prove that $J(B_n)(\mathbb{Q})_{\text{Tor}}\cong\mathbb{Z}/2\mathbb{Z}$ for all $n\leq30$. This follows from Theorem \ref{over finite} and the fact that $\gcd(5^{2^n}+1,13^{2^n}+1)=2$ for all $n\leq30$.    
\end{remark} 
In the Chebychev case, the positive density of primes at which $J(B_n)$ has supersingular reduction suggests the presence of latent symmetries. Indeed, one sees that $J(B_1)$ is an elliptic curve which has complex multiplication by $\mathbb{Z}[\sqrt{-2}]$: \[ [\sqrt{-2}](x,y)=\left(-1/2\cdot \frac{x^2-2}{x-2}+2\;,\; \frac{1}{-2\sqrt{-2}}\cdot\frac{y\cdot((x-2)^2-2)}{(x-2)^2}\right) .\]  

Moreover, by checking Igusa invariants against known examples, one sees that $J(B_2)$ also has complex multiplication by $\mathbb{Q}(\sqrt{\sqrt{2}-2})$. Statement $1$ of Theorem \ref{thm:Chebychev}, which follows from the work of Carocca, Lange, and Rodriguez \cite{CM}, shows that these examples are no accident and provides a construction of hyperelliptic curves, defined over the rational numbers, which have complex multiplication. 

There is much more we can prove in the Chebychev case, especially about statements pertaining to rational points. In fact, the technique which we use to determine $B_n(\mathbb{Q})$ for all $n\geq2$ generalizes to polynomials possessing a $\Gal(\bar{\mathbb{Q}}/\mathbb{Q})$-stable cycle, which we discuss after restating and proving the following consolidated theorem.
   
\Chebychev*
\begin{proof} For the first statement, note that $\mathfrak{C}_{n+1}$ defined in (\ref{aux curves}) has complex multiplication by $\mathbb{Q}(\zeta)$, induced by the map \[[\zeta]:\mathfrak{C}_{n+1}\rightarrow \mathfrak{C}_{n+1},\;\;\; [\zeta](x,y)=(\zeta^2x,\zeta y).\] We have already seen that the quotient curve of $\mathfrak{C}_{n+1}$ by the automorphism $\psi$ is $B_n$ and that $J(\mathfrak{C}_n)\sim J(B_n)^2$ over $\bar{\mathbb{Q}}$. It follows that the simple factors of $J(B_n)$ have complex multiplication by some subfield of $\mathbb{Q}(\zeta)$; see \cite[Theorem 3.3]{Lang}.  

In \cite[Theorem 2]{CM} it was shown that ${\mathfrak{C}_{n+1}'}:y^2=x(x^{2^{n+1}}-1)$ has a quotient $X$ with the property that $J(X)$ has complex multiplication by $\mathbb{Q}(\zeta+\zeta^d)$. Moreover, since $\mathbb{Q}(\zeta+\zeta^d)$ does not contain any proper CM fields, $J(X)$ must be absolutely simple; see \cite[Theorem 3.3]{Lang}. 

However, note that $\mathfrak{C}_{n+1}$ and $\mathfrak{C}_{n+1}'$ are twists, becoming isomorphic over $\mathbb{Q}(\zeta)$. Since any decomposition of an abelian variety is unique up to isogeny, we must have that $J(X)\sim J(B_n)$. Hence, $\End_0(J(X))\cong\End_0(J(B_n))$ and $J(B_n)$ has CM as claimed.

\begin{remark} \cite[Theorem 2]{CM} was established by studying the general case of metacyclic Galois coverings $Y\rightarrow\mathbb{P}^1$ branched at $3$-points, building upon previous work of Ellenberg. To translate, the relevant Galois covering group is 
\begin{equation}{\label{Galois covering}}
G=\Big\langle\,\,[\zeta],\; \psi\; \Big\vert \;\;\; [\zeta]^{2^{n+1}}=\psi^2=1,\; \psi\circ[\zeta]\circ\psi=[\zeta^d]\;\Big\rangle,
\end{equation} 
 and one can take $Y$ to be $\mathfrak{C}_{n+1}'$.   
\end{remark}   
 
The second statement of Theorem \ref{thm:Chebychev} is a restatement of Theorem \ref{over finite}. For the third statement, we use the fact that $J(B_n)(\mathbb{Q})_{\text{Tor}}$ injects into $J(B_n)(\mathbb{F}_p)$ via the reduction map \cite{Katz}. Hence, $J(B_n)(\mathbb{Q})_{\text{Tor}}\cong\mathbb{Z}/2\mathbb{Z}$ follows from the fact that \[\gcd\big(5^{2^n}+1,13^{2^n}+1,29^{2^n}+1,\dots\big)=2\;\;\text{for all}\;\;n,\] where the above set ranges over all (or almost all) primes $p\equiv5\bmod{8}$. To see this, fix $n$ and suppose that $p^*$ is an odd prime which divides $p^{2^n}+1$ for almost all $p\equiv5\bmod{8}$. In particular $p^{2^n}\equiv-1\bmod{p^*}$, and hence $p^*\equiv1\bmod{4}$. 

Now note that the $\gcd(4p^*+1, 8p^*)=1$, so that Dirichlet's theorem on arithmetic progressions implies that there exist infinitely many primes $p_0$ with $p_0=4p^*+1+8p^*k_0$ for some integer $k_0$. Moreover, since $4p^*+1\equiv5\bmod{8}$, we have that $p_0\equiv5\bmod{8}$. Hence, we may choose $p_0$ such that $p_0^{2^n}+1\equiv0\bmod{p^*}$ by our assumption on $p^*$.  

Finally, one sees that $1\equiv4p^*+1+8p^*k_0\equiv p_0\bmod{p^*}$ and $2\equiv p_0^{2^n}+1\equiv0\bmod{p^*}$. This is a contradiction since $p^*$ is odd.

On the other hand $f^n(0)=2$ for all $n\geq2$, from which it follows that $(0,2)\in B_n(\mathbb{Q})$. Moreover, by the argument above, this point (after embedding it into in the Jacobian) is not a torsion point. Hence, $J(B_n)(\mathbb{Q})$ has positive rank. The statement regarding the rank of the rational points of $J(C_n)$ follows from Proposition \ref{prop: decomp}.   
 
Finally, we prove statement $4$ of Thereom \ref{thm:Chebychev}. If $n\geq2$, then $C_n$ maps to $B_1:y^2=(x+2)(x^2-2)$. However, $B_1$ is an elliptic curve, and a $2$-descent shows that $B_1(\mathbb{Q})$ has rank zero. It follows that $B_1(\mathbb{Q})=\{\infty,(-2,0)\}$, and after computing preimages, we see that $C_n(\mathbb{Q})$ contains only the infinite points.

If $n=2$, a $2$-descent shows that the rank of $J(B_2)(\mathbb{Q})$ is one. Moreover, after running the Chabauty function in Magma \cite{Magma}, we see that $B_2(\mathbb{Q})=\{\infty,(-2,0),(0,\pm{2})\}$. This matches our claim for larger $n$. For the remaining $n\geq3$, we use covering collections to determine $B_n(\mathbb{Q})$. 

Since the resultant of $x+2$ and $f^n$ is equal to $2$ (for all $n$), the rational points on $B_n$ are covered by the rational points on the curves \[D_n^{(d)}:\;\; du^2=x+2,\;\;\;dv^2=f^n(x),\;\;\;\text{for}\; d\in\{\pm{1},\pm{2}\};\] see \cite[Example 9]{Stoll-R}. We will proceed by examining the second defining equation $C_n^{(d)}:dv^2=f^n(x)$ of $D_n^{(d)}$. If $d=1$, then our description of $C_n(\mathbb{Q})$ implies that $D_n(\mathbb{Q})$ has only the points at infinity. If $d=-2$, then $C_n^{(-2)}$ maps to the elliptic curve $B_1^{(-2)}:-2v^2=(x+2)\cdot(x^2-2)$ via $(x,y)\rightarrow(f^{n-1}(x),-2\cdot f^{n-2}(x)\cdot y)$. A descent shows that $B_1^{(-2)}(\mathbb{Q})$ has rank zero, from which it easily follows that $B_1^{(-2)}(\mathbb{Q})=\{\infty,(-2,0)\}$. By computing preimages, we find that $C_n^{(-2)}$ and $D_n^{(-2)}$ have no rational points. 

For the remaining cases when $d=-1$ and $d=2$, we map $C_n^{(d)}$ to $B_2^{(d)}$ via $(x,y)\rightarrow(f^{n-1}(x),d\cdot f^{n-2}(x)\cdot y)$. However, in either scenario, we find that $J(B_2^{(d)})(\mathbb{Q})$ has rank one, and moreover, we can compute a generator using bounds between the Weil and canonical heights. After running the Chabauty function in Magma \cite{Magma} and computing preimages, we find that 
\[C_n^{(2)}(\mathbb{Q})=\{(0,\pm{1}),(\pm{2},\pm{1})\}\;\; \text{and}\;\;C_n^{(-1)}(\mathbb{Q})=\{(\pm{1},\pm{1})\}\;\;  \text{for all}\;\;n\geq3.\]
Consequently, $D_n^{(2)}(\mathbb{Q})=\{(0,\pm{1},\pm{1}),(-2,0,\pm{1})\}$ and $D_n^{(-1)}(\mathbb{Q})=\emptyset$. Moreover, we see that $B_n(\mathbb{Q})=\{\infty,(-2,0),(0,\pm{2})\}$ as claimed. 
\end{proof}

Notice that we have determined the rational points on infinitely many curves $B_n$ each of which do not cover any lower genus curves (their Jacobians have complex multiplication by a CM field with no proper CM subfields). 

This is a normally difficult task. However, because $-2$ has finite orbit under application of the polynomial $x^2-2$, the rational points on $B_n$ are covered by finitely many computable twists of a curve $D_n$ . Moreover the twists are independent of $n$. Furthermore, each of the finitely many $D_n^{(d)}$ maps to many lower genus curves where standard rational points techniques may be applied (e.g. Chabauty's method, covering collections, rank zero Jacobians) more reasonably. 

We illustrate this situation in the following diagram.  

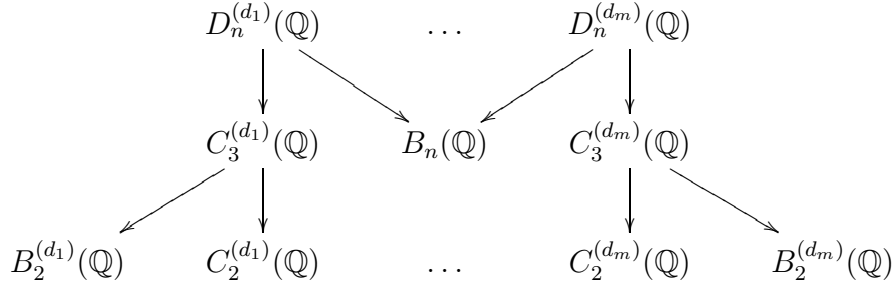
\begin{displaymath}
    \xymatrix{
       & D_n^{(d_1)}(\mathbb{Q}) \ar[d] \ar[dr] &\dots & D_n^{(d_m)}(\mathbb{Q})\ar[dl]\ar[d] \\
        &C_3^{(d_1)}(\mathbb{Q}) \ar[dl]\ar[d]     & B_n(\mathbb{Q})& C_3^{(d_m)}(\mathbb{Q})\ar[d]\ar[dr] \\
        B_2^{(d_1)}(\mathbb{Q})&C_2^{(d_1)}(\mathbb{Q})& \dots & C_2^{(d_m)}(\mathbb{Q})&B_2^{(d_m)}(\mathbb{Q})} 
\end{displaymath} 
\\
\indent A general way to construct examples of families of curves with this stability behavior is to use rational polynomials $f$ with a finite $\Gal(\bar{\mathbb{Q}}/\mathbb{Q})$-stable cycle. 

For instance, if $f(x)=x^2-31/48$, then $f$ has a $\Gal(\bar{\mathbb{Q}}/\mathbb{Q})$-stable $4$ -cycle:  
\begin{center}$\begin{CD}
1/4+\sqrt{-15}/6@>f>>-1+\sqrt{-15}/12 \\
@AAfA     @VVfV\\
-1-\sqrt{-15}/12 @<f<< 1/4-\sqrt{-15}/6.
\end{CD}$ \end{center}       

We construct a polynomial $g$ from this cycle. Set $\alpha=1/4+\sqrt{-15}/6$ and let $\bar{\alpha}=1/4-\sqrt{-15}/6$ be its Galois conjugate. Similarly, set $\beta=-1+\sqrt{-15}/12$ and $\bar{\beta}=-1+\sqrt{-15}/12$. Now, if one considers \[g(x)=x^2-x/2+23/48=(x-\alpha)(x-\bar{\alpha})\;\;\text{or}\;\; g(x)=x^2+2x+53/48=(x-\beta)(x-\bar{\beta}),\] then $\Supp(\Res(g,f^n))\subseteq\{2,3,23,53\}$ for either choice of $g$. 

Hence, to determine the rational points on $y^2=g(x)\cdot f^n(x)$ for all $n\geq1$, it suffices to compute the rational points on \[D_{g}^{(d)}(f^n):\;\; du^2=g(x),\;\;\;\; dv^2=f^n(x),\] where $d=\pm2^{e_0}\cdot3^{e_1}\cdot23^{e_2}\cdot53^{e_3}$ and $e_i\in\{0,1\}$ for all $i\geq0$. However, the equation $dv^2=f^n(x)$ maps to $dv^2=f^m(x)$ for all $m<n$, and a strategy for determining $D_{g}^{(d)}(f^n)(\mathbb{Q})$ is to choose an $m$ for which the Jacobian of $dv^2=f^m(x)$ has rank zero, or at least satisfies Chabauty's condition. 

One can use this to show that
\begin{equation}\label{fact} A_n:\;y^2=g(x)\cdot f^n(x)\;\; \;\text{satisfies}\;\; A_n(\mathbb{Q})=\{\infty^{\pm}\}\;\; \text{for all}\;\; n\geq1.
\end{equation} 

 Now that we have studied these curves defined by quadratic iteration in some detail, we use them and the theory of rational points on curves to classify Galois behavior in the dynamical setting.   
\section{3. Dynamical Galois Groups and Curves}{\label{Classifying Galois}}
 
\indent\indent In order to probe how the curves $C_n$ and their quadratic twists relate to the Galois theory of $f^n$, we must first discuss the necessary background. Let $f\in\mathbb{Q}[x]$ be a polynomial of degree $d$ whose iterates are separable. That is, we assume that the polynomials obtained from successive composition of $f$ have distinct roots in an algebraic closure.

 We fix some notation. Let $\mathbb{T}_n$ denote the set of roots of $f, f^2,\dots ,f^n$ together with $0$, and let $G_n(f)$ be the Galois group of $f^n$ over the rationals. Furthermore, set
\begin{equation}{\label{Arboreal}} \mathbb{T}_\infty:=\bigcup _{n \geq 0} f^{-n}(0)\;\;\text{and}\;\; G_\infty=\varprojlim G_n(f). 
\end{equation}  
Note that $\mathbb{T}_n$ (respectively $\mathbb{T}_\infty$) carries a natural $d$-ary rooted tree structure: $\alpha,\beta\in\mathbb{T}_n$ share an edge if and only if $f(\alpha) =\beta$. Moreover, as $f$ is a polynomial with rational coefficients, $G_n(f)$ acts via graph automorphisms on $\mathbb{T}_n$. Hence, we have injections $G_n \hookrightarrow \Aut(\mathbb{T}_n)$ and $G_\infty \hookrightarrow \Aut(\mathbb{T}_\infty)$. Such a framework is called the \emph{arboreal representation} associated to $f$ and we can ask about the size of the image $G_\infty\leq\Aut(T_\infty)$. For a nice exposition, see \cite{R.Jones}.  

\begin{remark} Note that $\Aut(\mathbb{T}_n)$ is the $n$-fold iterated wreath product of $\mathbb{Z}/d\mathbb{Z}$. We will use this characterization when useful. 
\end{remark}  

In the quadratic case, it has been conjectured that the image of $G_\infty(f)$ is ``large" under mild assumptions on $f$; see \cite[Conjecture 3.11]{R.Jones} for a more general statement. 
\begin{conjecture}[Finite-Index]{\label{Finite Index}} Let $f\in\mathbb{Q}[x]$ be a quadratic polynomial. If all iterates of $f$ are irreducible and $f$ is post-critically infinite, then $\big\vert\Aut(\mathbb{T}_\infty):G_\infty(f)\big\vert$ is finite. 
\end{conjecture} 

Here post-critically infinite means the orbit of the unique root of $f$'s derivative, also known as critical point, is infinite. This is an analog of Serre's result for the Galois action on the prime-powered torsion points of a non CM elliptic curve. For a discussion of this analogy, see \cite{B-J}.  

These hypotheses are easily satisfied in the case when $f(x)=x^2+c$ and $c$ is an integer: If $c\neq-2$ and $-c$ is not a square, then $f^n$ is irreducible for all $n$, and the the set $\{f(0),f^2(0),\dots\}$ is infinite. Stoll has given congruence relations on $c$ which ensure that the Galois groups of iterates of $f(x)=x^2+c$ are maximal $\big($i.e $G_n(f)\cong\Aut(\mathbb{T}_n)$ for all $n\big)$; see \cite{Stoll-Galois}. However, much is unknown as to the behavior of integer values not meeting these criteria, not to mention the more general setting of rational $c\;\big($for instance $c=3$ and $2/3\big)$.

In order to attack the finite index conjecture for more general values of $c$, we use the following fundamental lemma (due to Stoll), which gives a criterion for the maximality of the Galois groups of iterates in terms of rational points; see \cite[Corollary $1.3$]{Stoll-Galois} for the case when $f=x^2+c$ and $c$ is an integer or \cite[Lemma $3.2$]{Jones2} for the result concerning all rational quadratic polynomials:
 
\begin{lemma}{\label{Stoll-lemma}} Let $f\in\mathbb{Q}[x]$ be a quadratic polynomial, let $\gamma\in\mathbb{Q}$ be such that $f'(\gamma)=0$, and let $K_m$ be a splitting field for $f^m$. If $f,f^2,\dots f^n$ are all irreducible polynomials, then the subextension $K_n/K_{n-1}$ is not maximal if and only if $f^n(\gamma)$ is a square in $K_{n-1}$.   
\end{lemma} 
\begin{remark} Let $f(x)=x^2+c$. Note that with the hypotheses of Lemma \ref{Stoll-lemma}, $K_n/K_{n-1}$ is not maximal if and only in $(0,y)\in C_n(K_{n-1})$ for some $y\in K_{n-1}$. Furthermore, $|\mathbb{Q}(y):\mathbb{Q}|=2$ as $f\in\mathbb{Q}[x]$.   
\end{remark} 
As promised, we use quadratic twists of $C_n$ and the Hall-Lang conjecture on integral points of elliptic curves to prove the finite index conjecture \ref{Finite Index} in this case when $c$ is an integer. Before we restate and prove Theorem \ref{thm:Hall-Lang}, we remind the reader of the aforementioned conjectures. 

\begin{conjecture}[Hall]{\label{Hall}} For all $\epsilon>0$ there is a constant $C_\epsilon$ (depending only on $\epsilon$) such that for all nonzero $D\in\mathbb{Z}$ and all $x,y\in\mathbb{Z}$ satisfying $y^2=x^3+D$, we have $|x|\leq C_\epsilon\cdot D^{2+\epsilon}$.  
\end{conjecture} 
Conjecture \ref{Hall} is often referred to as the weak form of Hall's conjecture. The original conjecture was made with $\epsilon=0$, and though not yet disproven, this form is no longer believed to be true. Lang would later generalize Hall's conjecture to the following.  
\begin{conjecture}[Hall-Lang]{\label{Hall-Lang}} There are absolute constants $C$ and $\kappa$ such that for every elliptic curve $E/\mathbb{Q}$ given by a Weierstrass equation \[y^2=x^3+Ax+B\;\; \text{with}\; A,B\in\mathbb{Z}\] and for every integral point $P\in E(\mathbb{Q})$, we have that \[|x(P)|\leq C\cdot \max\{|A|,|B|\}^\kappa.\] 
\end{conjecture} 

See \cite[9.7]{Silv} for the relevant background material. Assuming these conjectures, we have the following dynamical corollaries.

\HallLang* 
\begin{proof} 
Let $f(x)=x^2+c$ and suppose that the subextension $K_n/K_{n-1}$ is not maximal. We will show that such an $n\geq2$ is bounded.

Since $c$ is an integer such that $-c$ is not a square, \cite[Corollary 1.3]{Stoll-Galois} implies that $f^n(0)$ is not a square, and that $f^1,f^2,\dots f^n$ are irreducible polynomials. Lemma \ref{Stoll-lemma} implies that $f^n(0)$ is a square in $K_{n-1}$. Hence, there is some $y\in\mathbb{Z}$ \;such that \[dy^2=f^n(0), \; \text{for}\;\;\mathbb{Q}(\sqrt{d})\subset K_{n-1}\;\text{and $d$ square-free}.\]   
Moreover, $d$ is a product of distinct primes $p_i$ dividing $2\cdot\prod_{j=1}^{n-1}f^j(0)$. To see this latter fact, we use the formula for the discriminant $\Delta_m$ of $f^m$ \[\Delta_m=\pm\Delta_{m-1}^2\cdot2^{2^m}\cdot f^m(0),\] given in Lemma 2.6 of \cite{Jones2}. It follows that the rational primes which ramify in $K_{n-1}$ must divide $2\cdot\prod_{j=1}^{n-1}f^j(0)$. Since the primes which divide $d$ must ramify in $K_{n-1}$, we obtain the desired description of the $p_i$. Also note that if $p_i|f^j(0)$ and $p_i|f^n(0)$ (which is the case since $p_i|d$), then it divides $f^{n-j}(0)$. 

In any case, we can have the refinement that $d=\prod_ip_i$, where the $p_i$'s are distinct primes dividing $2\cdot\prod_{j=1}^{\lfloor n/2\rfloor}f^j(0)$. Here $\lfloor x\rfloor$ denotes the floor function.   

A rational point on the curve $C^{(d)}_n:= dy^2=f^n(x)$ maps to $B_1^{(d)}:= dy^2=(x-c)\cdot f(x)$ via $(x,y)\rightarrow\big(f^{n-1}(x),\; y\cdot f^{n-2}(x)\big)$. Transforming $B_1^{(d)}$ into standard form, we get \[E_1^{(d)}:y^2=x^3+598752(c^2-3c)d^2x+161243136(c^3-18c^2)d^3,\]  

via $(x,y)\rightarrow\big(d\cdot(x-12c),\;2d^2y\big)$. In particular, if $K_n/K_{n-1}$ is not maximal, then we obtain an integer point 
\begin{equation}{\label{Integer point}}
\left(d\cdot(f^{n-1}(0)-12c),\; 2yd^2\cdot f^{n-2}(0)\right)\in E_1^{(d)}(\mathbb{Q}).
\end{equation}   
If we assume the Hall-Lang conjecture on integral points of elliptic curves, then there exist constants $C$ and $\kappa$ such that \[d\cdot(f^{n-1}(0)-12c)<C\cdot\max\{|598752(c^2-3c)d^2|, |161243136(c^3-18c^2)d^3|\}^\kappa.\]See \cite{Silv} or \cite{Coh} for the relevant background in elliptic curves. In either case 
\begin{equation}{\label{estimate1}}
|f^{n-1}(0)|<C'\cdot |d|^{\kappa'}\leq C'|f(0)\cdot f^2(0)\cdot\dots f^{\lfloor n/2\rfloor}(0)|^{\kappa'},
\end{equation} 
 for some new constants $C'$ and $\kappa'$. However, this implies that $n$ is bounded.

For example, if $c>0$, then $f^m(0)>f(0)\cdot f^2(0)\dots f^{m-1}(0)$ for all $m$. Hence, if we let $t=\lfloor n/2\rfloor+1$ and suppose $\kappa'<2^s$, then 
\begin{equation}{\label{estimate2}}
f^{n-1}(0)<C'\cdot (f^t(0))^{\kappa'}<C'(f^{t}(0))^{\kappa'}<C'f^{t+s}(0).
\end{equation} 
Since $0$ is not preperiodic, the result follows. A similar argument works in the case when $c\leq-3$, we simply use the fact that $|f^m(0)|\geq\big(f^{m-1}(0)-1\big)^2$; see \cite[Corollary 1.3]{Stoll-Galois} . 

Explicitly, when $c=3$ the $j$ invariant of $B_1^{(d)}$ is zero, and we may transform $B_1^{(d)}$ into the Mordell curve $M^{(-2d)^3}:= y^2=x^3-(2d)^3$. In particular, a point $(0,y)\in C^{(d)}_n(\mathbb{Q})$ yields a point \[\Big(\big(f^{n-1}(0)-1\big)\cdot d,\; d^2\cdot y\cdot f^{n-2}(0)\Big)\in M^{(-2d)^3}(\mathbb{Q}).\] 
If the weak form of Hall's conjecture for the Mordell Curves holds with $\epsilon=4$ and $C(\epsilon)=100$, then for $f(x)=x^2+3$ we have that \[ \Big\vert\big(f^{n-1}(0)-1\big)\cdot d\Big\vert<100\cdot|(-2d)^3|^6.\]
This implies that 
\begin{equation}{\label{estimate3}}
f^{n-1}(0)<26214400d^{17}+1\leq26214400\big(f^{\lfloor n/2\rfloor+1}(0)\big)^{17}+1.
\end{equation} 
However, such a bound implies that $n\leq13$. Moreover, one checks that the only $n\leq13$ with $dy^2=f^n(0)$ and $d$ equal to a product of distinct primes dividing $2\cdot\prod_{j=1}^{n-1}f^j(0)$, is $n=3$. In this case, $f^3(0)=7^2\cdot3=7^2\cdot f(0)$ and $|\Aut(\mathbb{T}_3):G_3(f)|=2$. It follows that the index of the entire family $|\Aut(\mathbb{T}_\infty):G_\infty(f)|$ must also equal $2$.  
\end{proof}  
\begin{remark}For $\epsilon=4$, our constant $C=100$ safely fits the data pitting the known integer points on the Mordell curves against the size of their defining coefficients; see \cite{Mordell}. In fact, even for Elkies' large examples \cite{Elkies}, our choice of $C$ works (thus far) for the strong form of Hall's conjecture ($\epsilon=0$), though our proof will work if one insists on a much larger constant.      
\end{remark} 
We now take a different approach to studying these dynamical Galois groups. In Theorem \ref{thm:Hall-Lang}, we used the curves coming from $f_c(x)=x^2+c$ to analyze the Galois theory of $f$'s iterates (viewing $c\in\mathbb{Q}$ as fixed), investigating the stability as $n$ grows. We now change our perspective slightly. Suppose that we fix a stage $n$, and ask for which rational values of $c$ is the Galois group of $f_c^n$ smaller than $\Aut(\mathbb{T}_n)$. Of course this question needs to be refined, as there will be many trivial values of $c$ (for instance if $-c$ is a square). A natural adjustment then is to ask for which rational numbers is $G_n(f_c)$ smaller than expected for the first time at stage $n$. As noted in Theorem \ref{thm:Hall-Lang}, an interesting example is $c=3$ and $n=3$:  
\begin{equation} \Gal((x^2+3)^2+3)\cong D_4\cong\Aut(\mathbb{T}_2),
\end{equation} 
yet one computes that $|\Gal(((x^2+3)^2+3)^2+3)|=64<2^{2^3-1}$, and hence the third iterate of $f=x^2+3$ is the first of $f$'s iterates to have a Galois group which is not maximal. This leads to the following definition:

\begin{definition} Let $c$ be a rational number and let $n\geq2$. If $f_c=x^2+c$ is a quadratic polynomial such that $G_{n-1}(f_c)\cong\Aut(\mathbb{T}_{n-1})$ and $G_n(f_c)\ncong\Aut(\mathbb{T}_n)$, then we say that $f_c$ has a newly small $n$-th iterate. Furthermore, let \[S^{(n)}:=\{c\in\mathbb{Q}\;|\,f_c\,\;\text{has a newly small $n$-th iterate}\},\] be the set of rational values of $c$ supplying a polynomial with a newly small $n$-th iterate.
\end{definition}      

Our refined question then becomes to describe $S^{(n)}$. In the case when $n=3$, the author completely characterizes $S^{(3)}$ in terms of the $x$-coordinates of two rank-one elliptic curves; see \cite[Theorem 3.1]{Me}. In particular, using bounds on linear forms in elliptic logarithms, we concluded that $S^{(3)}\cap\mathbb{Z}=\{3\}$; see \cite[Corollary 3.1]{Me}. Hence, $x^2+3$ is the only integer polynomial in the family $x^2+c$ with this particular Galois degeneracy.  

Furthermore, in \cite{Me} we use this characterization to compute many new examples of polynomials with newly small third iterate by adding together the points corresponding to known examples on the elliptic curve: e.g. $f(x)=x^2-2/3, x^2+6/19, x^2-17/14$. Moreover, since generators of both curves are easily computed and since the complement of $S^{(3)}$ and the points on the curves are explicit \cite[Theorem 3.1]{Me}, we have in some sense found all examples.   

It is natural to ask what happens for larger $n$. As an illustration, consider the case when $n=4$ and $c=-6/7$. One computes with Sage \cite{Sage} that  
\[\big\vert\Gal(((x^2-6/7)^2-6/7)^2-6/7)\big\vert=2^{2^3-1}\;\;\text{and}\;\;\ \big\vert\Gal((((x^2-6/7)^2-6/7)^2-6/7)^2-6/7)\big\vert=8192.\]
Since $8192<2^{2^4-1}$, we see that $-6/7\in S^{(4)}$. Are there any other examples? Although not entirely satisfactory, we have the following theorem which we restate from the introduction:      

\SmallFourth*
\begin{proof} 
As in \cite{Me}, we associate values in $S^{(n)}$ with the rational points on certain curves. To continue, we are in need of the following lemma: 
\begin{lemma}{\label{4th lemma}}Let $f_c(x)=f(x)=x^2+c$ and let $K_m$ be the splitting field of $f^m$ over $\mathbb{Q}$. If $c\in S^{(4)}$, then $f^4$ is irreducible and 
\begin{footnotesize}
\[\mathbb{Q}(\sqrt{-c}),\;\mathbb{Q}(\sqrt{f^2(0)}),\;\mathbb{Q}\left(\sqrt{-\frac{f^2(0)}{c}}\right), \;\mathbb{Q}(\sqrt{f^3(0)}),\;\mathbb{Q}\left(\sqrt{-\frac{f^3(0)}{c}}\right),\;\mathbb{Q}\left(\sqrt{\frac{f^3(0)}{f^2(0)}}\right),\;\mathbb{Q}\left(\sqrt{-\frac{f^3(0)}{c+1}}\right)\]
\end{footnotesize}
 are the distinct quadratic subfields of $K_3$. 

\begin{proof}First note that if $\Gal(f^m)\cong\Aut(\mathbb{T}_m)$, then $K_m$ contains exactly $2^m-1$ quadratic subfields. The reason is that the number of quadratic subfields is the number of subgroups of $\Gal(K_m)$ whose quotient is $\mathbb{Z}/2\mathbb{Z}$. Now $\Aut(\mathbb{T}_m)$ is the $m$-fold wreath product of $\mathbb{Z}/2\mathbb{Z}$, and one can show that the maximal abelian quotient of exponent $2$ of this group is $\mathbb{Z}/2\mathbb{Z}\oplus\dots\oplus\mathbb{Z}/2\mathbb{Z}$, where the product is $m$-fold; see \cite{Stoll-Galois}. This quotient group, by its maximality property, will contain as a subgroup any quotient that is abelian of exponent $2$, and hence the quotients that are isomorphic to $\mathbb{Z}/2\mathbb{Z}$ are in one-to-one correspondence with the subgroups of $(\mathbb{Z}/2\mathbb{Z})^m$ of order $2$. However, that is the same as the number of distinct elements of order $2$, which is $2^m - 1$. 

Now for the proof of Lemma \ref{4th lemma}. Suppose that $c\in S^{(4)}$. Then $\Gal(f^j)\cong\Aut(\mathbb{T}_j)$ for all $1\leq j\leq3$ and the subextensions $K_3/K_2$, $K_2/K_1$ and $K_1/\mathbb{Q}$ are all maximal. In particular $-f(0)=-c$ is not a rational square and $K_1=\mathbb{Q}(\sqrt{-c})$. Since $\Aut(\mathbb{T}_2)$ acts transitively on the roots of $f^2$, it follows that $f^2$ is irreducible. Then Lemma \ref{Stoll-lemma} implies that $f^2(0)\notin\ (K_1)^2$. In particular, the distinct quadratic subfields of $K_2$ are 
\begin{equation}{\label{2nd subfields}}
\mathbb{Q}(\sqrt{-c}),\; \mathbb{Q}(\sqrt{f^2(0)})\;\; \text{and}\; \mathbb{Q}\left(\sqrt{\frac{f^2(0)}{-c}}\;\right).
\end{equation} 
 Note that when $m\geq2$, then discriminant formula in (\ref{discriminant}) implies that $\sqrt{f^m(0)}\in K_m$. By the opening remarks in the proof of Lemma \ref{4th lemma}, there must be exactly $2^2-1$ or $3$ such subfields. Hence our list is exhaustive for $K_2$.  

One simply repeats this argument for the third iterate, obtaining the claimed list of quadratic subfields of $K_3$. In particular the set $\{-f(0),f^2(0),f^3(0)\}$ does not contain a rational square. It suffices to show that $f^4(0)$ is also not a rational square to deduce that $f^4$ is irreducible; see \cite{Jones2}. However, a $2$-descent on the curve \[F_0: y^2=f_c^4(0)=((c^2+c)^2+c)^2+c\] shows that its Jacobian has rank zero. After reducing modulo several primes of good reduction, one finds that any torsion must be $2$-torsion. Hence $F_0(\mathbb{Q})=\{\infty^{\pm},(0,0),(-1,0)\}$, and $c\in S^{(4)}$ implies that $c\neq0,-1$.       
\end{proof} 
\end{lemma}
With Lemma \ref{4th lemma} in place, we are ready to relate the elements of $S^{(4)}$ with the rational points on certain curves. If $c\in S^{(4)}$ then $f_c^4$ is irreducible and $f_c^4(0)$ is not a rational square (see the proof of the Lemma \ref{4th lemma}). However, since the extension $K_4/K_3$ is not maximal by assumption, Lemma \ref{Stoll-lemma} implies that $\sqrt{f_c^4(0)}\in\ K_3$. Hence $\sqrt{f_c^4(0)}$ must live in one of the seven quadratic subfields of $K_3$ listed in Lemma \ref{4th lemma}. Thus, there must exist $y\in\mathbb{Q}$ such that $(c,y)$ is a rational point on one of the following curves: 
\begin{align*}
F_1: y^2&=\frac{f_x^4(0)}{-x}=-(x^7 + 4x^6 + 6x^5 + 6x^4 + 5x^3 + 2x^2 + x + 1),\\
F_2: y^2&=\frac{f_x^4(0)}{f_x^2(0)}=x^6 + 3x^5 + 3x^4 + 3x^3 + 2x^2 + 1,\\
F_3: y^2&=\frac{f_x^4(0)}{-(x+1)}=-x(x^6 + 3x^5 + 3x^4 + 3x^3 + 2x^2 + 1),\\
F_4: y^2&=\frac{f_x^4(0)}{x}\cdot\frac{f_x^3(0)}{x}=(x^7 + 4x^6 + 6x^5 + 6x^4 + 5x^3 + 2x^2 + x + 1)\cdot(x^3 + 2x^2 + x + 1),\\
F_5: y^2&= f_x^4(0)\cdot\frac{f_x^3(0)}{-x}=(x^8 + 4x^7 + 6x^6 + 6x^5 + 5x^4 + 2x^3 + x^2 + x)\cdot-(x^3 + 2x^2 + x + 1),\\
F_6: y^2&=\frac{f_x^4(0)}{f_x^2(0)}\cdot f_x^3(0)=(x^6 + 3x^5 + 3x^4 + 3x^3 + 2x^2 + 1)\cdot(x^4 + 2x^3 + x^2 + x),\\
F_7: y^2&=\frac{f_x^4(0)}{-(f_x^2(0))}\cdot \frac{f_x^3(0)}{x}=-(x^6 + 3x^5 + 3x^4 + 3x^3 + 2x^2 + 1)\cdot(x^3 + 2x^2 + x + 1),
\end{align*} (after dividing out the finite singularities coming from $c=0,-1$ when necessary).

Note that all of these curves are hyperelliptic, and so at least in principal, their arithmetic: ranks, integer points, etc. are more easily computable. Also note that the interesting rational points corresponding to known elements of $S^{(4)}$ both come from $F_2(\mathbb{Q})$:
\begin{equation}{\label{point list}}
\big\{\infty^+,\infty^-,(0,\pm{1}),(2/3,\pm{53/27}),(-6/7,\pm{377/343})\big\}\subseteq F_2(\mathbb{Q})
\end{equation} 
Therefore, to describe $S^{(4)}$ it suffices to characterize $F_i(\mathbb{Q})$ for all $1\leq i\leq7$. We will do this sequentially, employing standard methods in the theory of rational points on curves. For a nice overview of these techniques, see \cite{Stoll-R}. 

\textbf{Case 1:} A 2-descent with Magma  \cite{Magma} shows that the rational points on the Jacobian of $F_1$ have rank one. Moreover, reducing modulo various primes of good reduction, one sees that the order of any rational torsion point must divide $4$. However, the only $2$-torsion point is $Q=[(-1,0)-\infty]$ and by examining the image of $Q$ via the $2$-descent map, one sees that $Q$ is not a double in $J(F_1)(\mathbb{Q})$. It follows that $J(F_1)(\mathbb{Q})\cong\mathbb{Z}/2\mathbb{Z}\oplus\mathbb{Z}$. Since $F_1$ has genus three, we can apply the method of Chabauty and Coleman to bound the rational points (in fact, find them all).  

To do this we change variables to obtain an equation which is more amenable to the computations to come: send $(x,y)\rightarrow(-x-2,y)$ to map to the curve \[F'_1: y^2=x^7+10x^6+42x^5+94x^4+117x^3+76x^2+21x+1.\] A naive point search yields $\{\infty, (-1,0), (0,\pm{1})\}\subseteq F_1'(\mathbb{Q})$ and we will show this set is exhaustive. 

Let $J=J(F_1')$ and use the point $P_0=(0,1)$ to define an embedding of $F_1'(\mathbb{Q})\subseteq J(\mathbb{Q})\subseteq J(\mathbb{Q}_3)$ via $P\rightarrow[P-P_0]$. Then given a $1$-form $\omega$ on $J(\mathbb{Q}_3)$, one can integrate to form the function \[J(\mathbb{Q}_3)\rightarrow\mathbb{Q}_3,\;\;\;\;\text{given by}\;\;\;\; P\rightarrow\int_{0}^P\omega.\] Coleman's idea was to notice that if we restrict this function to a residue class of $F_1'(\mathbb{Q}_3)\subset J(\mathbb{Q}_3)$, then this function can be computed explicitly in terms of power series (using a parametrization coming from a uniformizer for the class). Moreover, in the case when the genus is larger than the rank of the group $J(\mathbb{Q})$, as is the case in our example, one can find an $\omega$ where the above function vanishes on $J(\mathbb{Q})$. Finally, using Newton polygons, one can bound the number of rational points in each residue class by bounding the number of zeros of a power series in $\mathbb{Z}_3$. This is what we will do. For a nice exposition on this method, see \cite{Poonen}.      

We will follow the notation and outline of Example $1$ in Section $9$ of Wetherell's thesis \cite{Wetherell}. In particular, we use $x$ as a local coordinate system on the residue class at $(0,1)$ and  the basis $\eta_0=(1/y)dx, \eta_1=(x/y)dx$, and $\eta_2=(x^2/y)dx$ for the global forms on $F_1'$. Expanding $1/y$ in a power series in terms of $x$ we get:  
\[\eta_0=\frac{dx}{y}=1 -\frac{ 21}{2}x + \frac{1019}{8}x^2 - \frac{28089}{16}x^3 + \frac{3292019}{128}x^4 - \frac{99637707}{256}x^5 + \frac{6153979535}{1024}x^6\dots\]
Furthermore, it is known that the $\eta_i$ are in $\mathbb{Z}_3[[x]]$. Then we have the integrals $\lambda_i$ for the $\eta_i$ in the residue class of $(0,1)$; that is 
\[\lambda_i(P)=\int_{(0,1)}^P\eta_i.\] 
From our formulas for the $\eta_i$, we have: 
\begin{small}\[\lambda_0=x-\frac{21}{4}x^2+\frac{1019}{24}x^3-\frac{28089}{64}x^4+\frac{3292019}{640}x^5-\frac{33212569}{512}x^6+\frac{6153979535}{7168}x^7-\dots\] 
\[\lambda_1=\frac{1}{2}x^2-\frac{7}{2}x^3+\frac{1019}{32}x^4-\frac{28089}{80}x^5+\frac{3292019}{768}x^6-\frac{99637707}{1792}x^7+ \frac{6153979535}{8192}x^8 -\dots \]
\[\lambda_2=\frac{1}{3}x^3 - \frac{21}{8}x^4 + \frac{1019}{40}x^5 - \frac{9363}{32}x^6 + \frac{3292019}{896}x^7 - \frac{99637707}{2048}x^8+ \frac{6153979535}{9216}x^9+\dots\]\end{small}   
Let $\omega_i$ be the differentials on $J$ corresponding to the $\eta_i$ on $F_1'$, i.e. the pullbacks relative to the inclusion $F_1'(\mathbb{Q}_3)\subseteq J(\mathbb{Q}_3)$ given by $P\rightarrow[P-(0,1)]$. Finally, let $\lambda_i'$ be the homomorphism from $J(\mathbb{Q}_3)$ to $\mathbb{Q}_3$ obtained by integrating the $\omega_i$. We will calculate the $\lambda_i'$ on $J_1(\mathbb{Q}_3)$, the kernel of the reduction map. 

Let $a\in J_1(\mathbb{Q}_3)$, so that $a$ may be represented as $a=[P_1+P_2+P_3-3P_0]$ with $P_i\in C(\bar{\mathbb{Q}}_3)$ and $\widebar{P_i}=\widebar{P_0}=(0,1)$. If $s_j=\sum_{i=1}^{3}x(P_i)^j$, then from the expression \[\int_0^a\omega_i=\sum_j\int_{(0,1)}^{P_j}\eta_i\] we see that 
\begin{small} \[\lambda_0=s_1-\frac{21}{4}s_2+\frac{1019}{24}s_3-\frac{28089}{64}s_4+\frac{3292019}{640}s_5-\frac{33212569}{512}s_6+\frac{6153979535}{7168}s_7-\dots\] 
\[\lambda_1=\frac{1}{2}s_2-\frac{7}{2}s_3+\frac{1019}{32}s_4-\frac{28089}{80}s_5+\frac{3292019}{768}s_6-\frac{99637707}{1792}s_7+ \frac{6153979535}{8192}s_8 -\dots \]
\[\lambda_2=\frac{1}{3}s_3 - \frac{21}{8}s_4 + \frac{1019}{40}s_5 - \frac{9363}{32}s_6 + \frac{3292019}{896}s_7 - \frac{99637707}{2048}s_8+ \frac{6153979535}{9216}s_9+\dots\]\end{small} 
We wish to find an $\omega$ whose integral kills $J(\mathbb{Q})$. However, since $\log(J(\mathbb{Q}))$ has rank one in $T_0(J(\mathbb{Q}_3))$, the dimension of such differentials is $2$, and so we have some freedom with our choice. We will exploit this freedom to bound the number of points in each residue field. 

Note that if $U=[\infty-P_0]$, then $12U$ is in $J_1(\mathbb{Q}_3)$. Furthermore, a $1$-form kills $J(\mathbb{Q})$ if and only if it kills $12U$, since the index of the subgroup generated by the rational torsion and $U$ is coprime to $|J(\mathbb{F}_3)|=24$. Using Magma, we calculate the divisor $12U$ represented as $[P_1+P_2+P_3-3P_0]$ where the $3$ symmetric functions in the $x(P_i)$ are          
 \begin{small}\begin{align*}
\sigma_1&=\frac{5688167583876464940561144764011383197382945288}{5528939601706074645413409528185601232466043121}\equiv2\cdot3^4\bmod{3^5},\\
\\
\sigma_2&=\frac{- 2183647192786560140353830791558556354713308560}{5528939601706074645413409528185601232466043121}\equiv2^2\cdot3^3\bmod{3^5},\\
\\
\sigma_3&=\frac{4352156372570507181684433225178910249832181376}{5528939601706074645413409528185601232466043121}\equiv2^3\cdot3^3\bmod{3^5}.
\end{align*}
\end{small} 
Choosing a precision of $3^5$ was arbitrary, though sufficient for our purposes. Note that the valuation of every $x(P_i)$ is at least $\min\{v(\sigma_i)/i\}=\min\{4,3/2,1\}=1$. It follows that $v(s_j)\geq j$. Moreover, one verifies that every term past $j=3$ of $\lambda_i(12U)$ is congruent to $0\bmod{3^5}$. 
 
After calculating the $s_j$ in terms of the $\sigma_i$, one has that 
\begin{align*}
\lambda_0(12U)&\equiv 2\cdot3^3\bmod{3^5},\\      
\lambda_1(12U)&\equiv3^3\bmod{3^5},\\ 
\lambda_2(12U)&\equiv2^3\cdot3^3 \bmod{3^5}.
\end{align*} 
As the integral is linear in the integrand, there exist global one forms $\alpha$ and $\beta$ such that \[\int_0^P\alpha=0\;\;\;\; \text{and}\;\;\;\;\;\int_0^P\beta=0\;\;\;\;\;\text{for all}\;\;\;P\in \J(\mathbb{Q}),\] with $\alpha\equiv2\lambda_1-\lambda_0\bmod{3^5}$ and $\beta\equiv2\lambda_2-4\lambda_0\bmod{3^5}$. Moreover, we have that $\alpha=(x-1)dx/y$ and $\beta=(x^2-1)dx/y$ when we view them over $\mathbb{F}_3$. However, for every $P\in F_1'(\mathbb{F}_3)$, either $\alpha$ or $\beta$ does not vanish at $P$; see the table below.  
\begin{center}
  \begin{tabular}{|c| c | c| }
    \hline
    $P\in F_1'(\mathbb{F}_p)$& $ord_P(\alpha)$&$ord_P(\beta)$\\ \hline
    $\widebar{\infty}$ & 2 & 0 \\ \hline
    $\widebar{(0,1)}$ & 0 & 0 \\ \hline
    $\widebar{(0,-1)}$ & 0 & 0 \\ \hline
    $\widebar{(-1,0)}$&0&1\\
    \hline
  \end{tabular}
\end{center}
It follows that every residue class in $F_1'(\mathbb{Q}_3)$ contains at most one rational point, and hence exactly one rational point as claimed.

\textbf{Case 2:} We first use Runge's method to find $F_2(\mathbb{Z})$. This involves completing the square. Suppose we have an integer solution $x$. We rewrite our equation as \[y^2-(x^3+3/2x^2+3/8x+15/16)^2 = -61/64x^2 - 45/64x + 31/256,\] and then multiply by $256$ to clear denominators. Write 
\[g=16x^3 + 24x^2 + 6x + 15,\;\;\;\;
h= -244x^2 - 180x + 31,\;\;\;\;
Y=16y.\] Then $(Y-g)(Y+g)=h$, and (unless one of the factors is zero) $|Y-g| \le |h|$ and $|Y+g| \le |h|$. Note that neither factor can be zero since $h$ has no integer roots. After combining our inequalities, we see that 
\begin{equation}{\label{Runge}}
| 2g | = | (Y+g) - (Y-g)| \le 2 |h|.
\end{equation} 
Hence $|g| \le |h|$. As the degree of $g$ is larger than $h$, we get a small bound on $x$. A naive point search shows that $F_2(\mathbb{Z})=\{\infty^{\pm},(-1,\pm{1}),(0,\pm{1}),(-2,\pm{1})\}$. 

As for the full rational points, this is (at the moment) beyond reach; for an explanation, see Remark \ref{difficulty} at the end of this section. At best, we can be sure that the unknown rational points must be of very large height (on the order of $10^{100}$) by running the Mordell-Weil sieve. Since our curve is of genus $2$, we can use explicit bounds between the Weil and canonical heights to compute a basis of $J(F_2)(\mathbb{Q})$. It follows that $J(F_2)(\mathbb{Q})$ has basis \begin{small}\[ P_1=[(0,-1)+(-1,-1)-\infty^--\infty^+], P_2=[(0,-1)+(0,-1)-\infty^--\infty^+], P_3=[(0,1)+(-2,1)-\infty^--\infty^+],\]\end{small} which we use while sieving with Magma; see \cite{Sieve} for a full discussion of the Mordell-Weil sieve, or \cite{Stoll-R} for a basic introduction;            

\textbf{Case 3:} For $F_3$ and subsequent curves, we use unramified covers to determine the rational points. Note that the resultant $\Res(x^6 + 3x^5 + 3x^4 + 3x^3 + 2x^2 + 1,-x)=1$,  and we study the curves \[D^{(d)}: \;du^2= x^6 + 3x^5 + 3x^4 + 3x^3 + 2x^2 + 1, \;\;\;dv^2=-x\;\;\;d\in\{\pm1\},\] which are $\mathbb{Z}/2\mathbb{Z}$-covers of $F_3$. Moreover, every rational point on $F_3$ lifts to one on some $D^{(d)}$; see \cite[Example 9]{Stoll-R}. If $d=-1$, then the curve $du^2= x^6 + 3x^5 + 3x^4 + 3x^3 + 2x^2 + 1$ has no rational points, since it has no points in $\mathbb{F}_3$. Hence $D^{(-1)}$ has no rational points. On the other hand, if $d=1$, then our description of $F_2(\mathbb{Z})$ shows that $F_3(\mathbb{Z})=\{\infty, (-1,\pm{1}), (0,0)\}$. Moreover, if we assume that there are no unknown points in $F_2(\mathbb{Q})$, then we have found all of the rational points on $F_3$. In any case $0,-1\notin S^{(4)}$. Hence, $F_3$ contributes no integers to $S^{(4)}$. 

\textbf{Case 4:} Similarly for $F_4$ we have the covers \[D^{(d)}: du^2=(x+1)\cdot(x^3 + 2x^2 + x + 1),\;\;\; dv^2=x^6 + 3x^5 + 3x^4 + 3x^3 + 2x^2 + 1,\;\;\;d\in\{\pm{1}\}.\] Again, if $d=-1$, then one sees that there are no rational points on $D^{(d)}$ by looking in $\mathbb{F}_3$. When $d=1$, the second defining equation of $D$ is that of $F_2$, and so we use our description of $F_2(\mathbb{Z})$ to show that $F_4(\mathbb{Z})=\{\infty^{\pm},(-1,0),(0,\pm{1}),(-2,\pm{1})\}$. Under the assumption that there are no unknown rational points on $F_2$, we can conclude that $F_4(\mathbb{Z})=F_4(\mathbb{Q})$. In any event $0,-1,-2\notin S^{(4)}$, and so $F_4$ also contributes nothing to $S^{(4)}$.   

\textbf{Case 5:} The Jacobian of $F_5$ has rank zero, and we easily determine that $F_5(\mathbb{Q})=\{\infty^{\pm}, (0,0)\}$. Since $0\notin S^{(4)}$, we conclude that $F_5$ contributes no integers to $S^{(4)}$. 

\textbf{Case 6:} Since $\Res(x^6 + 3x^5 + 3x^4 + 3x^3 + 2x^2 + 1, x^4+2x^3+x^2+x)=-1$, the rational points on $F_6$ are covered by the points on \[D^{(d)}: du^2=x^6 + 3x^5 + 3x^4 + 3x^3 + 2x^2 + 1,\;\;\; dv^2=x^4+2x^3+x^2+x,\;\;\;d\in\{\pm{1}\}.\] If $d=1$, then the equation $u^2=x^4+2x^3+x^2+x$ is an elliptic curve of rank zero having rational points $\{\infty^{\pm}, (0,0)\}$. It follows that $D^{(1)}(\mathbb{Q})$ has only the points at infinity and those corresponding to $x=0$. If $d=-1$, then $D^{(-1)}$ has no points over $\mathbb{F}_3$. We conclude that $F_6(\mathbb{Q})=\{\infty^{\pm},(0,0)\}$ unconditionally, and nothing new is added to $S^{(4)}$ in this case.  

\textbf{Case 7:} The rational points on the final curve $F_7$ are covered by the points on \[D^{(d)}: du^2=x^6 + 3x^5 + 3x^4 + 3x^3 + 2x^2 + 1,\;\;\; dv^2=-(x^3+2x^2+x+1),\;\;\;d\in\{\pm{1}\}.\] As in previous cases, if $d=-1$, then $D^{(-1)}$ has no points over $\mathbb{F}_3$. On the other hand, when $d=1$ we use our description of $F_2(\mathbb{Z})$ and  $F_2(\mathbb{Q})$ do determine that $F_7(\mathbb{Z})=\{\infty, (-2,\pm{1})\}$. Like before, if there are no unknown rational points on $F_2$, then we will have given a complete list of points on $F_7$. Note that $-2\notin S^{(4)}$ and so, in combination with the previous cases, we have shown that $S^{(4)}\cap\mathbb{Z}=\emptyset$.        
 
Finally, if $c\neq3$ is an integer such that $G_2(f_c)\cong\Aut(T_2)$, then $G_3(f_c)\cong\Aut(T_3)$ by \cite[Corollary 3.1]{Me}. However, since $S^{(4)}\cap\mathbb{Z}=\emptyset$, we deduce that $G_4(f_c)\cong\Aut(T_4)$ also.             
\end{proof}
\begin{remark}{\label{difficulty}} At the moment, proving that we have determined $F_2(\mathbb{Q})$ is far beyond reach. For one, the Galois group of $x^6 + 3x^5 + 3x^4 + 3x^3 + 2x^2 + 1$ is S$_6$, from which it follows that $\End(J(F_2))\cong\mathbb{Z}$; see \cite{Zarhin}. In particular, $F_2$ does not map to any lower genus (elliptic) curves. Morevover, the rank of $J(F_2)(\mathbb{Q})$ is $3$, and so the method of Chabauty and Coleman is not applicable. Furthermore, in order to use a covering collection coming from the pullback of the multiplication by $2$ map on the Jacobian, one would need to determine generators of the Mordell-Weil group $E(K)$, where $E$ is an elliptic curve defined over a number field $K$ of degree $17$, followed by Elliptic Chabauty. This is currently not feasible.\end{remark}       

A naive point search on the relevant curves suggests that $S^{(5)}$ and $S^{(6)}$ are probably empty. In fact, if $n=5$, then the $15$ curves, corresponding to the $2^4-1$ quadratic subfields of $K_4$, satisfy the Chabauty condition (small rank), and so proving that $S^{(5)}$ is empty may be doable. This begs the question as to whether all $n$ sufficiently large satisfy $S^{(n)}=\emptyset$. Is it true for all $n\geq5$?

This seems beyond reach at the moment. However, the weaker statement that $S^{(n)}\cap\mathbb{Z}=\emptyset$ may be attackable if one assumes standard conjectures on the height of integer points on hyperelliptic curves relative to the size of the defining coefficients. If true, this would amount to a nice Galois uniformity principle. Namely, outside of the small exception $x^2+3$, it seems as though $G_2(f_c)\cong\Aut(\mathbb{T}_2)$ already implies $G_n(f_c)\cong\Aut(\mathbb{T}_n)$ for all $n\geq1$.     
\\
\\
\textbf{Acknowledgements:} It is a pleasure to thank my advisor, Joe Silverman, as well as Michael Stoll, Rafe Jones, Samir Siksek, Reinier Broker, Daniel Hermes, and Vivian Olsiewski Healey for the many useful discussions while this work was in progress. I also thank the referee for their helpful comments.

\begin{center}\textbf{Mailing Adress:} 
\\
Wade Hindes \\
Department of Mathematics, Brown University\\
Box $1917$\\
$151$ Thayer Street\\
 Providence, RI $02912$\end{center} 
\end{document}